\documentclass{amsart}
\usepackage{amsmath, latexsym, amssymb, amsthm, multirow, tikz, float, longtable}
\usetikzlibrary{calc,matrix,patterns,shadings,backgrounds}

\theoremstyle{definition}

\theoremstyle{plain}
\newtheorem{theorem}{Theorem}
\newtheorem{lemma}{Lemma}

\newtheorem{proposition}{Proposition}

\newtheorem{conjecture}{Conjecture}

\title[New Examples of Torsion-Free Non-unique Product Groups]{New Examples of Torsion-Free Non-unique Product Groups}
\author[W.Carter]{William Carter}

\begin{document}

\begin{abstract}
 We give an infinite family of torsion-free groups that do not satisfy the unique product property.  For these examples, we also show that each group contains arbitrarily large sets whose square has no uniquely represented element.
\end{abstract}

\maketitle

\section{Introduction}
If $k[G]$ is a group ring over a torsion-free group, two natural questions that can be asked are what are the zero divisors, and what are the units? Both questions are very well known and considered to be two of the least tractable questions in the theory of group rings. A detailed discussion of the history of these problems (and other interesting open questions) can be found in \cite{MR798076}. 

\begin{conjecture}
\emph{ Zero Divisor Conjecture (Kaplansky)}
\label{zdc}
If $G$ is a torsion-free group and $K$ is an integral domain, then the group ring $K[G]$ has no zero divisors.
\end{conjecture}

Similarly, the second conjecture, which implies Conjecture \ref{zdc}, can be stated as.
\begin{conjecture}
\emph{Nontrivial Units Conjecture (Kaplansky)}
\label{nuc}
If $G$ is a torsion-free group and $K$ is a field, then the only units in $K[G]$ are the trivial ones, i.e. those of the form $kg$ where $k\in K-\{0\}$ and $g\in G$.
\end{conjecture}

The unique product property was initially conceived as an attempt to solve these conjectures. A group is $G$ is said to satisfy the \emph{unique product property} if given any two non-empty finite sets $X,\;Y \subset G$ then at least one element, say $z$ in the product set $XY=\{xy \mid x\in X$ and $y\in Y \}$ can be written uniquely as a product, $z=xy$ where $x\in X$ and $y\in Y$. Finite product sets within these groups are studied in \cite{MR2975160}. Many familiar groups satisfy this property, for example, orderable groups \cite{MR798076}, diffuse groups \cite{MR1794287} and locally indicable groups \cite{MR0310046}. Moreover, it is well known that every right orderable group satisfies this property. The converse, however, is still open.  

Any group with torsion does not satisfy the unique product property, so the only interesting examples of groups without this property would necessarily be torsion-free.
There are only two known examples of torsion-free groups that do not satisfy the unique product property (excluding, of course, torsion-free groups that contain either of these two examples as a proper subgroup). 
 
The first example was given by E. Rips and and Y. Segev. The authors showed that there exists a family of torsion-free groups that do not satisfy this property \cite{MR887195} ( See \cite{2013arXiv1307.0981S}, for a recent treatment of this family of groups). In their examples, given predetermined sets, relations for a group were carefully constructed that in such a way that the resulting group is torsion-free and contains the two sets as a pair of non-unique product sets. Many seemingly natural questions regarding these groups are still open. In particular, nothing is known about these groups in relation to Conjectures \ref{zdc} or \ref{nuc}.  

The second known example of a group that does not satisfy the unique product property and the only known explicit example of such a group was given by D. Promislow in \cite{MR940281}. By means of a random search algorithm, he found a $14$ element set $S$ in the group 
\[P = \langle x, y \mid xy^2x^{-1}y^2,\; yx^2y^{-1}x^2 \rangle\]
 with the property that $SS$ has no uniquely represented element. We will call such a set $S$ a \emph{non-unique product set}. Given the nature of the search, very little is known about other non-unique product sets in $P$ or about how to extend this result to other groups. 
 
A result due to Lewin, \cite{MR0292957}, shows that $P$ satisfies Conjecture \ref{zdc}.
\begin{theorem}
\emph{(Lewin) }
\label{Lewin}
If $G= G_1 *_{G_N} G_2$ a free product with amalgamation, where
\begin{enumerate}
\item $G_N$ is normal in both $G_1$ and $G_2$;
\item $F[G_1]$ and $F[G_2]$ have no zero divisors;
\item $F[G_N]$ satisfies the Ore condition. 
\end{enumerate}
Then $F[G]$ has no zero divisors.
\end{theorem} 

To see this, note that $P\cong K*_{\mathbb{Z}^2}K$, where $K$ is a Klein bottle group and we identify index $2$ subgroups that are isomorphic to $\mathbb{Z}^2$ in each copy of $K$. The second condition holds since torsion-free one relator groups are locally indicable \cite{MR741011} and group rings over locally indicable groups satisfy Conjecture \ref{zdc}. For the last condition, it is well known that a group ring over an abelian group satisfies the Ore condition. It is still unknown whether $P$ satisfies Conjecture \ref{nuc}.

The purpose of this paper is to generate new simple examples of groups that do not satisfy the unique product property and to produce non-unique product sets whose existence can be inferred from the relations in the group.  Currently, it is not all together clear where to look for such groups or even sets within these groups. All that is currently known is that these groups must be non-left orderable. In fact, this is precisely why $P$ was initially seen as a likely candidate \cite{MR536897}; however, this does not tell us how to find such sets or even if they exist (clearly, any finite pair of subsets will not work). The hope is that generating more examples will lead to a better understanding of the structure of such groups. In Section $4$, we do so by generalizing $P$ in the following way.

\begin{theorem}
\label{nup}
For each $k>0$, the torsion-free group
\[P_k=\langle a, b, \mid ab^{2^k}a^{-1}b^{2^k},\; ba^2b^{-1}a^2 \rangle \]
does not satisfy the unique product property, and for $k>1$, does not contain $P$. 
\end{theorem}
Note that the group $P_1$ is the same as Promisow's example $P$.
The relations of $P_1$ and $P_k$ are similar, but the groups are quite different. For example, it is well known that $P$ is a finite extension of $\mathbb{Z}^3$ and as such is supersolvable. In contrast, the groups $P_k$ for $k>1$ are much larger. One can show $P_k$ contains a finite index subgroup isomorphic to $\mathbb{Z}^2 \times F$, where $F$ is a finitely generated free group. In particular, these groups are also not amenable and hence are not solvable.  An argument, identical to the one above, shows that each $P_k$ satisfies the hypotheses of Theorem \ref{Lewin} and thus every group $P_k$ satisfies Conjecture \ref{zdc}. 

These groups are generalizations of $P$ in the sense that each $P_k$ is is an amalgamation of Klein bottle groups over $\mathbb{Z}^2$. However, we wish to emphasize that the non-unique product sets we construct in Section $4$ are not generalizations of Promislow's set $S$ found \cite{MR940281}, but rather arise from a careful study of the geometry of the Cayley graph given by the presentation above.  Roughly, the idea is to construct specific paths in the Cayley graph taken sufficiently long so that the Klein bottle relations force certain paths from the product set to overlap nicely. In Section $5$, this idea is extended to longer paths in the Cayley graph to prove the following result. 

\begin{theorem}
\label{unbound}
Each group $P_k$ contains arbitrarily large non-unique product sets.
\end{theorem}

\section{Preliminaries}
If a group $G$ acts by automorphisms on a simplicial tree $T$ without inversion (that is, no element of $G$ exchanges the endpoints of an edge $e$), then $T$ is called a $G$-tree. The action is said to be \emph{trivial} if $G$ fixes a point and \emph{minimal} if there is no invariant $G$-subtree except for $T$ itself. 

In this setting, an automorphism is said to be \emph{elliptic} if it fixes a point and \emph{hyperbolic} otherwise. If $g$ is elliptic, we define $Fix(g)$ to be the set of all points fixed by $g$. Following \cite{MR1954121}, we can characterize these automorphisms in the following way.
\begin{proposition}
\label{bass serre}
Let $G$ be  group that acts on a simplicial tree $T$ by automorphisms without inversion.
\begin{enumerate}
\item If $g\in G$, then either $g$ acts on a unique simplicial line in $T$ by translations or $Fix(g) \neq \emptyset$.
\item If $g_1,\; g_2 \in G$ and $Fix(g_1),\; Fix(g_2)$ are nonempty and disjoint, then $Fix(g_1g_2) = \emptyset$.
\item If $G$ is generated by a finite set of elements $s_1,\;s_2,\; \dots,\; s_m$ such that $s_j$ and $s_is_j$ fix points in $T$ for all $i$, $j$, then the action of $G$ is trivial. 
\end{enumerate}
\end{proposition}

The unique simplicial line in $(1)$ is called the \emph{axis} of $g$ and denoted $A_g$.
Further, following \cite{MR907233}, we can describe minimal subtrees in the following way.
\begin{proposition}
\label{cm tree}
If $G$ is finitely generated and $T$ is a non-trivial $G$-tree, then $T$ contains a unique minimal $G$-invariant subtree, which is the union of the axes of all the hyperbolic elements in $G$.
\end{proposition} 

A natural setting for groups acting on $G$-trees is when $G$ splits as a free product with amalgamation, an HNN extension, or more generally as the fundamental group of a graph of groups. From \cite{MR1954121} there exists a tree $T$, referred to as the Bass-Serre tree, on which $G$ acts simplicially. For our purposes, we need only consider the case in which $G\cong A*_C B$. In this case, such a tree is described as follows. The vertices of the tree $T$ are given by $G/A \cup G/B$. The edges are given by $G/C$, with initial vertices $v_i( gC)= gA$ and the terminal vertices $v_t( gC)= gB$.
The group $G$ acts on $T$ on the left. The stabilizers of the vertices are the conjugates of $A$ and $B$, and the edge stabilizers are the conjugates of $C$.

\section{Properties of the Groups $P_k$}
 
Note that just as in $P$, each group $P_k$ is a free product with amalgamation. To see this, fix $k>0$, and take two Klein bottle groups 
\[ K_1 = \langle a, \; x \mid axa^{-1}x \rangle  
\text{ and } 
 K_2 = \langle y, \; b \mid byb^{-1}y \rangle \] 
with subgroups 
\[A_1 = \langle a^{2}, \; x \rangle \cong \mathbb{Z}^2
\text{ and } 
A_2 = \langle b^{2^k}, \; y \rangle \cong \mathbb{Z}^2.\]
respectively.
If we define the isomorphism 
\[\phi: A_1 \rightarrow A_2\text{ by } x \mapsto b^{2^k}\text{ and } a^{2} \mapsto y,\]
 then the free product of $K_1$ and $K_2$ with amalgamation of $A_1$ and $A_2$, by $\phi$ has the presentation
\[ K_1*_{A_1}K_2 \cong \langle a,\; b,\; x,\; y \mid  axa^{-1}x,\; byb^{-1}y,\; x=b^{2^k} y=a^2 \rangle \cong P_k. \] 

For concreteness, we will choose transversal
\[ T_{K_1} = \{1,\; a\} \text{ and } T_{K_2} = \{1,\; b,\; \dots, \; b^{2^k-1}\}. \]
So, as an amalgamated product with transversal $T_{K_1}$ we have the following results.

\begin{proposition}
\emph{(Normal Forms)}
\label{normal forms}

Every element $w\in P_k$ can be written uniquely in the form:
\[w= a^{2u}b^{2^kv}a^{\alpha}b^{\beta_1}ab^{\beta_2}a\dots b^{\beta_l}ab^{\beta}\]
 where $u, \; v \in \mathbb{Z}$, $\alpha \in \{0,\; 1\}$, $\beta_i\in \{1,\; b,\; \dots, \; b^{2^k-1}\}$, and  $\beta\in \{0,\; 1,\; b,\; \dots, \; b^{2^k-1}\}$
\end{proposition}

As an amalgamated product of torsion-free groups, from \cite{MR1954121} we have
\begin{proposition}
\label{torsion theorem}
Every group $P_k$ is torsion-free.  
\end{proposition}

 Ultimately, we want to show that every group $P_k$ does not satisfy the unique product property and hence gives an infinite family of simple concrete examples. One issue that needs to be addressed is that some of the groups $P_k$ ($k>1$) could contain $P$ and hence not be truly new examples. We will show that every group does not contain $P$. This will be done by showing the following:
\begin{itemize}

\item If $A$, $B \in P_k$ where $\langle A,\; B\rangle$ fixes a line $L$ in $P_k$, and $\langle A,\; B\rangle$ acts on $L$ with no global fixed point, then the relations 
\[AB^2A^{-1}B^2=1 \text{ and } BA^2B^{-1}A^2=1 \]
can not simultaneously hold in $P_k$.
\item If $P \leq P_k$, then the induced action of $P$ on $P_k$ fixes a line $L_k$ in $T_k$.
\end{itemize}

\begin{lemma}
\label{lemma_1}
Suppose $\langle A, \; B\rangle$ fixes a line $L$ in $T_k$. If $A$ and $B$ are hyperbolic, then neither of the relations 
\[AB^2A^{-1}B^2=1 \text{ and } BA^2B^{-1}A^2=1 \]
can hold in $P_k$.
\end{lemma}

\begin{proof}
Suppose $A$ and $B$ are hyperbolic elements that stabilize the same line $L$. Then there are $m,\;n\in \mathbb{Z}$ so that $A^nB^{-m}$ fixes $L$ pointwise. So $A^nB^{-m}\in \langle a^2,\; b^{2^k}\rangle$ or rather $A^n = a^{2s_1}b^{2^kt_1} B^m$, for some $s_1,\; t_1 \in \mathbb{Z}$. By assumption, the relation $BA^2B^{-1}A^2=1$ holds and so the relation
\[1=BA^{2n}B^{-1}A^{2n}= B(a^{2s_1}b^{2^kt_1} B^m)^{2}B^{-1}(a^{2s_1}b^{2^kt_1} B^m)^{2}
= a^{2s_2}b^{2^kt_2} B^{4m}\]
also holds. 
 It follows then that $B^{4m}\in \langle a^2,\; b^{2^k}\rangle$, contradicting the fact that $B$ is hyperbolic. A similar result holds if we assume that $AB^2A^{-1}B^2=1$ holds.
\end{proof}

\begin{lemma}
\label{lemma_2}
 If $A$ is hyperbolic  and $B$ is elliptic, then the following relations 
\[AB^2A^{-1}B^2=1 \text{ and } BA^2B^{-1}A^2=1 \]
can not simultaneously hold in $P_k$.
\end{lemma}

\begin{proof}

Suppose otherwise. Note that every elliptic element in $P_k \cong K_1*_{A_1}K_2$ is conjugate to a word in $K_1$ or $K_2$, so conjugating if necessary, we may assume that 
\[ B=a^{2s}b^{2^kt}a^{r_1} \text{ or } B=a^{2s}b^{2^kt}b^{r_2}\]
where $r_1\in \{0,\; 1\}$ or $r_2 \in \{0,\; 1,\;\dots,\; 2^k-1\}$. From Proposition \ref{normal forms} we may write 
 \[A = a^{2u}b^{2^kv}a^{\alpha}b^{\beta_1}ab^{\beta_2}a \dots b^{\beta_l}ab^{\beta}\] as a reduced word in $P_k$ and since $A$ is hyperbolic, the subword  
\[a^{\alpha}b^{\beta_1}ab^{\beta_2}a \dots b^{\beta_l}ab^{\beta}\] 
contains non-trivial $a$ and $b$ subwords. Note that we need only consider 
\[ B=a^{2s}b^{2^kt}a \text{ or } B=a^{2s}b^{2^kt}b^{2^{k-1}}.\]
 Indeed, if either $r_1$ or $r_2$ is $0$, then the second relation implies that $A$ has finite order since $a^{2s}b^{2^kt}$ commutes with any square in $P_k$. Also, if $r_2 \neq 2^{k-1}$, then the right hand side of $1= AB^2A^{-1}B^2$ can be rewritten as
\[(a^{2u_1}b^{2^kv_1} )(a^{\alpha}b^{\beta_1} \dots b^{\beta_l}a) 
      (b^{r_3}) 
      (a b^{2^k-\beta_l} \dots ab^{2^k-\beta_1}a^{\alpha})
      (b^{r_3});
      \]  
a reduced word in normal form since $r_3= 2r_2\;(mod\; 2^k) \neq 0$. This violates uniqueness in Proposition \ref{normal forms}.  

  In either case of $B$, the idea of the proof is to analyze the possible values of
$\alpha,\; \beta_1,\; \beta_2,\; \dots,\; \beta_l,\;$ and $\beta$, and show that no such word $A$ exists.  

Consider the case $B= a^{2s}b^{2^k}a$.
The first relation
says that
\[ 1= AB^{2}A^{-1}B^{2} = a^{4s+2+\sigma_b(A)(4s+2)} \]
 which is true if and only if $\sigma_b(A)= -1$, where 
\[
   \sigma_b(A) = 
     \begin{cases}
1 & \text{if the sum of all the powers of }b \text{ in } A \text{ is even} \\
-1 & \text{if the sum of all the powers of }b \text{ in } A \text{  odd}. \\
\end{cases} 
\]

Suppose the relation 
\begin{equation} \label{case_1}
1=BA^2B^{-1}A^2= a^{2q}b^{2^kr} a(a^{\alpha}b^{\beta_1}a \dots b^{\beta_l}ab^{\beta})^2a^{-1}(a^{\alpha}b^{\beta_1}a \dots b^{\beta_l}ab^{\beta})^2
\end{equation}
holds. By assumption, $A$ is a hyperbolic element, and so $A^{2} \notin \langle a^2, \; b^{2^k} \rangle $. We claim that cancellation must occur in the subword $ab^{\beta} a^{-1}a^{\alpha}b^{\beta_1}$.
Otherwise, say in the case where $\alpha=0$ and $\beta\neq 0$, then the right hand side of \eqref{case_1}
 above can be written as a non-trivial word in normal form contradicting Proposition \ref{normal forms}.  
 Similarly, in the case where $\alpha=1$ and $\beta= 0$, the right hand side of \eqref{case_1} reduces to a non-trivial word in normal form, which also contradicts Proposition \ref{normal forms}. Hence, the only cases that need to be considered are when $\alpha=0$ and $\beta =0$ or when $\alpha=1$ and $\beta \neq0$. We will handle both cases at the same time, so for concreteness, relabel $\beta= \beta_{l+1}$. After reduction of the pair $aa^{-1}$, right hand side of \eqref{case_1} contains a subword of the form $b^{\beta_{i}+\beta_{j}}$. If $\beta_{i}+\beta_{j}=2^k$, move $b^{\beta_{i}+\beta_{j}}$ and the resulting $a^2$ to the far left in \eqref{case_1} as described by Proposition \ref{normal forms}. Repeat this process for the next resulting subword $b^{\beta_{i-1}+\beta_{j+1}}$. If at any stage of the reduction, we have $b^{\beta_s+\beta_t}\neq 2^k$, then the  reduced word in \eqref{case_1} is a non-trivial word in normal form, leading to a contradiction of Proposition \ref{normal forms}. Pairing off the powers of $b$ in this way,  
we have either:
\begin{enumerate}
 \item $\alpha =0$, $\beta=0$, $\beta_l+\beta_1=2^k$, $\beta_{l-1}+\beta_2=2^k$, $\dots$, $\beta_{\frac{l}{2}+1}+\beta_{\frac{l}{2}}=2^k$ (if $l$ is even),
 \item $\alpha =0$, $\beta=0$, $\beta_l+\beta_1=2^k$, $\beta_{l-1}+\beta_2=2^k$, $\dots$, $\beta_{\frac{l+1}{2}}+\beta_{\frac{l+1}{2}}=2^k$ (if $l$ is odd),
\item $\alpha =1$, $\beta\neq 0$, $\beta+\beta_1=2^k$, $\beta_{l}+\beta_2=2^k$, $\dots$, $\beta_{\frac{l+2}{2}}+\beta_{\frac{l+2}{2}}=2^k$ (if $l$ is even), or
\item $\alpha =1$, $\beta\neq 0$, $\beta+\beta_1=2^k$, $\beta_{l}+\beta_2=2^k$, $\dots$, $\beta_{\frac{l+1}{2}+1}+\beta_{\frac{l+1}{2}}=2^k$ (if $l$ is odd),
\end{enumerate}
In any event, this forces $\sigma_b(A)=1$ giving a contradiction.

Consider the other case, where $B=a^{2s}b^{2^kt}b^{2^{k-1}}$.  
Using the same normal form for $A$ as above, the relation
\[1= AB^{2}A^{-1}B^{2} = a^{4s+\sigma_b(A)4s}b^{2^{k+1}t+2^k+\sigma_a(A)(2^{k+1}t+2^k)}\]
 holds provided $\sigma_a(A)=-1$ and either $\sigma_b(A)=-1$ or $s=0$,
 where  
\[
   \sigma_a(w) = 
     \begin{cases}
1 & \text{if the sum of all the powers of }a \text{ in } w \text{ is even} \\
-1 & \text{if the sum of all the powers of }a \text{ in } w \text{  odd} \\
\end{cases}
\] 
 and $\sigma_b(A)$ is as above.
 
 An argument similar to the one above applied to the relation  
 \[1=BA^2B^{-1}A^2\]
 shows 
 \begin{enumerate}
 \item $\alpha =0$, $\beta \neq 0$, $\beta+\beta_1=2^{k-1}$, $\beta_l+\beta_2=2^k$,  $\dots$, $\beta_{\frac{l+1}{2}+1}+\beta_{\frac{l+1}{2}}=2^k$ ,
 \item $\alpha=0$, $\beta= 0$, $\beta_1=2^{k-1}$, $\beta_{l}+\beta_2=2^k$, $\dots$, $\beta_{\frac{l+1}{2}}+\beta_{\frac{l+1}{2}+1}=2^k$, or
  \item $\alpha =1$, $\beta=2^{k-1}$, $\beta_l+\beta_1=2^k$, $\beta_{l-1}+\beta_2=2^k$, $\dots$, $\beta_{\frac{l}{2}}+\beta_{\frac{l}{2}+1}=2^k$. 
\end{enumerate}
and so in every case, $\sigma_b(A)=1$.
 
 So we must have that $s=0$. If we simply count the number of exponents in $a$ of 
  $BA^2B^{-1}A^2$, one checks that after all possible cancellations, this is $8u + 4(2j+1)$ for some integer $j$, i.e. this is true by our description of $A$ and $B$ if no cancellations occur and any cancellation reduces the total number of exponents in $a$ by $8$. Since $8u + 4(2j+1)=0$ has no integer solution, this relation holding would contradict Proposition \ref{torsion theorem}.  
\end{proof}

\begin{lemma}
\label{lemma_3}
If $\langle A,\; B \rangle \subset P_k$ fixes some line $L$ in the Bass-Serre Tree $T_k$ where $A$ and $B$ are elliptic elements  with disjoint fixed point sets, then the following relations 
\[AB^2A^{-1}B^2=1 \text{ and } BA^2B^{-1}A^2=1 \]
can not simultaneously hold in $P_k$. 
\end{lemma}

\begin{proof}
If $A$ and $B$ are elliptic elements with disjoint fixed point sets, then by Proposition \ref{bass serre}, $AB$ acts by translation on some line in $T_k$. Moreover, $\langle AB,\; B\rangle = \langle A,\; B\rangle$ and if $A$ and $B$ satisfy the relations above, then so do $AB$ and $B$. So $\langle AB,\; B\rangle$ satisfies the hypotheses of the preceding lemma and both relations which contradicts the preceding lemma.
\end{proof}

\begin{theorem}
If $k_1$ and $k_2$ are distinct natural numbers, then $P_{k_1}$ is not isomorphic to $P_{k_2}$ and for $k>1$, $P_k$ does not contain $P$.
\end{theorem}

\begin{proof}
The first claim is clear as their abelianizations are not isomorphic. For the second claim, fix $k>1$ and suppose that $\langle A,\;B\rangle \cong P$ is a subgroup of $P_k$. Since $P_k$ acts on the Bass-Serre tree $T_k$, there is an induced action of $P$ on $T_k$ by isometries without edge inversion. It follows that the action of $P$ on $T_k$ has no global fixed point; otherwise, $P \leq K_1^{g}$ or $P \leq K_2^{g}$ for some $g\in P_k$ and in particular, this implies that the surface groups $K_1^{g}$ or $K_2^{g}$ contain a free Abelian group of rank $3$. Since $P$ is finitely generated and $T_k$ is non-trivial, by Proposition \ref{cm tree}, $T_k$ contains a unique minimal $P$-invariant subtree which we will denote by $L$. By Proposition \ref{bass serre}, $L$ contains at least one axis. On the other hand, since $P$ is a finite extension of $\mathbb{Z}^3$, the largest tree $P$ can act on is a line. So, if $P$ is a subgroup of $P_k$, we can deduce that $P$ acts simplicially on a line $L\subset T_k$. Applying Lemmas \ref{lemma_1}, \ref{lemma_2}, and \ref{lemma_3} gives us the desired contradiction.
 \end{proof}
We cannot say for certain whether the groups $P_k$, for $k>1$, contain one another.

\section{Family of groups}
Let $k$ be a fixed positive integer that we will use for the remainder of the paper. In this section, we will show that  $P_k$ does not satisfy the unique product property. Recall, that given a torsion-free group $G$, a subset of the form $\{xr^i \mid l \leq i \leq m\}$ for some $x$, $r\in G$ and $l,m\in \mathbb{Z}$ is said to be a \emph{left progression of ratio $r$}, or simply a \emph{left $r$-progression}. In $P_k$, consider the following $b$-progressions
\[X_0 = \{a^{-1}, \; a^{-1}b\},\]
\[X_i = \{b^ia^{-1}b^j \mid \; 0 \leq j \leq 2^k+1 \},\]
\[Y_l = \{b^lab^j \mid \; 1 \leq j \leq 2^k+1 \},\]
\[Z_0   = \{b^j \mid \; -2^k \leq j \leq 2^{k} \} \]
where  $1\leq i \leq 2^k-1$ and $0\leq l \leq 2^k-1$. Set 
\[T = \bigcup_{i=0}^{2^k-1} X_i \cup \bigcup_{j=0}^{2^k-1}Y_j \cup Z_0\]
  and for convenience, set $X=\bigcup_{i=0}^{2^k-1} X_i$ and  $Y = \bigcup_{j=0}^{2^k-1}Y_j$.
Proposition \ref{normal forms} shows every element in $T$ is distinct we will show that every element in $TT$ has no unique representation as follows. First, decompose $TT$ into smaller product sets of the form 
\[X_iX_j,\; Y_iX_j,\; X_iY_j,\; Y_iY_j,\; Z_0X_i,\; X_iZ_0,\; Z_0Y_i,\; Y_iZ_0,\; 
\text{ and }Z_0Z_0. \]
  
From there, we decompose these product sets further into progressions that are obtained as the product of single element in $T$ with one of the sets $X_i$, $Y_j$, or $Z_0$, which we will refer to as \emph{slices}. 

Showing $TT$ is a non-unique product set requires careful bookkeeping to make keeping track of the specific slices easier, we will adopt the following conventions. Write $x_{(n,m)}= b^n a^{-1} b^{m}$, $y_{(n,m)}= b^n a b^{m}$, and $z_{(0,n)} = b^n$ and if $u_{(m, i)} \in T$ and $W_n = \{w_{(n,j)}\mid l_n\leq j \leq m_n \}$ is one of our $b$-progressions listed above, we will denote the slices by 
\[u_{(m, i)}W_n = \{u_{(m, i)}w_{(n,j)} \mid l_n\leq j \leq  m_n\}.\] 

Clearly, any product in $TT$ that belongs to two of these slices has two different representations in $TT$. Using our choice of the $b$-progressions, we can efficiently show most of these slices are contained in at least one other slice. This reduces the number of elements we need to check to a much smaller set. For the remaining slices, the Klein bottle relations are used to show the remaining slices are contained in at least two of the subproduct sets listed above and hence have two distinct representations.

\subsection{Matching Common Words in the Progressions.\\} 
The following equalities and containments hold for subproduct sets in $TT$ as a result of the structure of the progressions. These are perhaps easiest to see visually, as in figures \ref{fig:T1}, \ref{fig:T2}, and \ref{fig:T3}, by writing the respective products $U_iX$, $U_iY$ and   $U_i Z_0$ in table form, where $U_i$ is an arbitrary progression in $T$. 
In figures \ref{fig:T1} and \ref{fig:T2}, the rows are labeled by individual words in a progression (written in order from the starting value $u_{i, s}$ to the ending value $u_{i, e}$) and the columns are labeled by the progressions in $X$ and $Y$ respectively.
In figure \ref{fig:T3}, both row and column are labeled by words in the respective progressions (also written in the order of the progression). In each the figures, the circled slices are those that are not paired up by the structure of the progressions mentioned above.\\

\textbf{Case $1$:} Consider products of the form $U_i Y$. As illustrated in figure \ref{fig:T1}, the slices along the diagonal lines are equal since we always have
\[   u_{(i, v+1)} Y_u = \{ b^i a^{\epsilon} b^{v+1} b^u a b^j \mid 1 \leq j \leq 2^k+1 \} = u_{(i, v)} Y_{u+1}, \]
where $\epsilon \in \{-1,\; 0,\; 1\}$ and $u$ and $v$ are taken in the appropriate range.
So the only slices we need consider separately, are those of the form
\[ u_{(i, s)}Y_0 \text{ and } u_{(i, e)} Y_{2^k-1} \]
for appropriate starting values $s$ and ending values $e$ of each progression.

\begin{figure}[h!]
\begin{tikzpicture}[description/.style={fill=white,inner sep=2pt}]
\matrix (m) [matrix of math nodes, row sep=.5em,
column sep=1em, text height=1.5ex, text depth=0.25ex]
{ \; & \; & Y_0 & Y_1 & Y_2 & \dots & Y_{2^k-2} & Y_{2^k-1}\\
u_{(i, s)} & \; & * & * & * & \dots & * & *\\ 
u_{(i, s+1)} & \; & * & * & * & \dots & * & *\\ 
u_{(i, s+2)} & \; & * & * & * & \dots & * & *\\
\vdots & \; & \vdots & \vdots & \vdots & \; & \vdots & \vdots\\
u_{(i, e-1)} & \; & * & * & * & \dots & * & *\\
u_{(i, e)} & \; & * & * & * & \dots & * & *\\
};

\path[-,font=\scriptsize]
(m-1-1.south) edge[auto] node[auto] {$ \; $} (m-1-8.south);

\path[-,font=\scriptsize](m-2-4) edge node[description]  {$ = $} (m-3-3);
\path[-,font=\scriptsize](m-3-4) edge node[description]  {$ = $} (m-4-3);
\path[-,font=\scriptsize](m-6-4) edge node[description]  {$ = $} (m-7-3);

\path[-,font=\scriptsize](m-2-5) edge node[description]  {$ = $} (m-3-4);
\path[-,font=\scriptsize](m-3-5) edge node[description]  {$ = $} (m-4-4);
\path[-,font=\scriptsize](m-6-5) edge node[description]  {$ = $} (m-7-4);

\path[-,font=\scriptsize](m-2-8) edge node[description]  {$ = $} (m-3-7);
\path[-,font=\scriptsize](m-3-8) edge node[description]  {$ = $} (m-4-7);
\path[-,font=\scriptsize](m-6-8) edge node[description]  {$ = $} (m-7-7);

\draw[thick,black] (m-1-2.north) -- (m-7-2.south);

\draw (m-2-3) circle [radius=.26];
\draw (m-7-8) circle [radius=.26];
\end{tikzpicture}
\caption{Matching Patterns for Products of the Form $U_i Y$} \label{fig:T1}
\end{figure} 

\textbf{Case $2$} Consider products of the form $U_i X$. Just as in Case $1$, we have similar identifications along the diagonal lines for all the slices with the same cardinality, as illustrated in figure \ref{fig:T2}. However, we also have proper containments, since the slices $u_{(i, j)}X_0$ only have cardinality $2$. There are two containments of particular interest, namely $u_{(i, s)}X_0 \subset u_{(i, s+1)}X_{2^k-1}$ and $u_{(i, s+1)}X_0 \subset u_{(i, s)}X_{1}$. Note that if $0\leq m\leq 2^k+1$, then 
\[u_{(i,s+1)}x_{(2^k-1,m)}=(b^ia^{\epsilon}b^{s+1})(b^{2^k-1}a^{-1}b^m)= 
b^ia^{\epsilon}b^{s}a^{-1}b^{m-2^k}\]  
is a consequence of the defining relation $ab^{2^k}a^{-1}b^{2^k} =1$.
After reindexing, it follows that former containment always occurs since 
\[u_{(i, s)}X_0 = \{ b^ia^{\epsilon}b^s a^{-1} b^j \mid j=0,\; 1 \} \subset \{b^i a^\epsilon b^s a^{-1} b^{j}\mid -2^k\leq j \leq 1\} = u_{(i, s+1)} X_{2^k-1}. \]
Containment in the latter case is clear, but it is worth mentioning this containment plays a very important role, later. The only slices we need to consider separately are those of the form 
\[ u_{(i, e)} X_{2^k-1} \text{ and the shortened } u_{(i, s)}X_1 \text{ written as } \{ u_{(i, s)} b a^{-1} b^j \mid 2\leq j \leq 2^k+1 \},\]
where once again $s$ and $e$ are the appropriate starting and ending values of the progression $U_i$.

\begin{figure}[h!]
\begin{tikzpicture}[description/.style={fill=white,inner sep=2pt}]
\matrix (m) [matrix of math nodes, row sep=.5em,
column sep=1em, text height=1.5ex, text depth=0.25ex]
{ \; & \; & X_0 & X_1 & X_2 & \dots & X_{2^k-2} & X_{2^k-1}\\
u_{(i, s)} & \; & * & * & * & \dots & * & *\\ 
u_{(i, s+1)} & \; & * & * & * & \dots & * & *\\ 
u_{(i, s+2)} & \; & * & * & * & \dots & * & *\\
\vdots & \; & \vdots & \vdots & \vdots & \; & \vdots & \vdots\\
u_{(i, e-1)} & \; & * & * & * & \dots & * & *\\
u_{(i, e)} & \; & * & * & * & \dots & * & *\\
};

\path[-,font=\scriptsize]
(m-1-1.south) edge[auto] node[auto] {$ \; $} (m-1-8.south);

\path[-,font=\scriptsize](m-2-4) edge node[description]  {$ \subset $} (m-3-3);
\path[-,font=\scriptsize](m-3-4) edge node[description]  {$ \subset $} (m-4-3);
\path[-,font=\scriptsize](m-6-4) edge node[description]  {$ \subset $} (m-7-3);

\path[-,font=\scriptsize](m-2-5) edge node[description]  {$ = $} (m-3-4);
\path[-,font=\scriptsize](m-3-5) edge node[description]  {$ = $} (m-4-4);
\path[-,font=\scriptsize](m-6-5) edge node[description]  {$ = $} (m-7-4);

\path[-,font=\scriptsize](m-2-8) edge node[description]  {$ = $} (m-3-7);
\path[-,font=\scriptsize](m-3-8) edge node[description]  {$ = $} (m-4-7);
\path[-,font=\scriptsize](m-6-8) edge node[description]  {$ = $} (m-7-7);

\path[-,font=\scriptsize] (m-2-3) edge[bend left=60] node[description]  {$\subset$} (m-3-8);

\draw[thick,black] (m-1-2.north) -- (m-7-2.south);

\draw (m-2-4) circle [radius=.26];
\draw (m-7-8) circle [radius=.26];
\end{tikzpicture}
\caption{Matching Patterns for Products of the Form $U_i X$ } \label{fig:T2}
\end{figure}

\begin{figure}[h!]
\begin{tikzpicture}[description/.style={fill=white,inner sep=2pt}]
\matrix (m) [matrix of math nodes, row sep=.5em,
column sep=.5em, text height=1.5ex, text depth=0.25ex]
{ \; & \; & z_{(0,-2^k)} & z_{(0,-2^k+1)} & z_{(0,-2^k+2)} & \dots & z_{(0,2^k-1)} & z_{(0,2^k)}\\
u_{(i, s)} & \; & * & * & * & \dots & * & *\\ 
u_{(i, s+1)} & \; & * & * & * & \dots & * & *\\ 
u_{(i, s+2)} & \; & * & * & * & \dots & * & *\\
\vdots & \; & \vdots & \vdots & \vdots & \; & \vdots & \vdots\\
u_{(i, e-1)} & \; & * & * & * & \dots & * & *\\
u_{(i, e)} & \; & * & * & * & \dots & * & *\\
};

\path[-,font=\scriptsize]
(m-1-1.south) edge[midway] node[auto] {$ \; $} (m-1-8.south);

\path[-,font=\scriptsize](m-2-4) edge node[description]  {$ = $} (m-3-3);
\path[-,font=\scriptsize](m-3-4) edge node[description]  {$ = $} (m-4-3);
\path[-,font=\scriptsize](m-6-4) edge node[description]  {$ = $} (m-7-3);

\path[-,font=\scriptsize](m-2-5) edge node[description]  {$ = $} (m-3-4);
\path[-,font=\scriptsize](m-3-5) edge node[description]  {$ = $} (m-4-4);
\path[-,font=\scriptsize](m-6-5) edge node[description]  {$ = $} (m-7-4);

\path[-,font=\scriptsize](m-2-8) edge node[description]  {$ = $} (m-3-7);
\path[-,font=\scriptsize](m-3-8) edge node[description]  {$ = $} (m-4-7);
\path[-,font=\scriptsize](m-6-8) edge node[description]  {$ = $} (m-7-7);

\draw[thick,black] (m-1-2.north) -- (m-7-2.south);

\draw (m-2-3) circle [radius=.26];
\draw (m-7-8) circle [radius=.26];
\end{tikzpicture}
\caption{Matching Patterns for Products of the Form $U_i Z_0$} \label{fig:T3}
\end{figure}
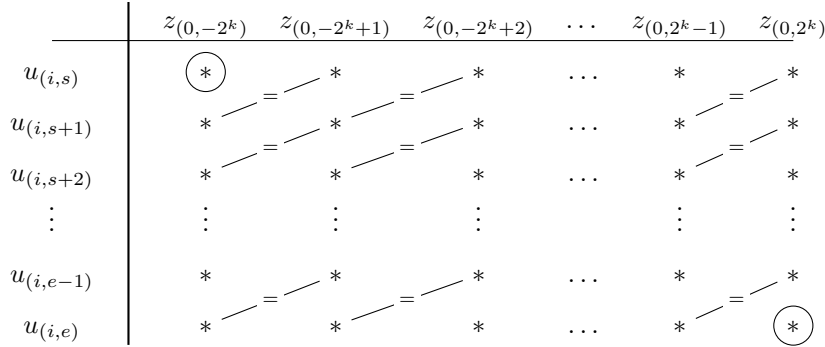

\textbf{Case 3:} Consider products of the form $U_iZ_0$. The equalities illustrated in the figure above are clear and in this case, the circled slices are matched as follows.
As illustrated in figure \ref{fig:T3}, each product has exactly two elements $\{u_{(i, s)}b^{-2^k},\; u_{(i, e)}b^{2^k} \}$ that are not identified within the table. If $U_i \neq Z_0$, then it is clear that 
\[ u_{(i, s)}b^{-2^k} = b^{2^k}u_{(i, s)} \subset Z_0U_i \text{ and } u_{(i, s)}b^{2^k} = b^{-2^k}u_{(i, s)}
\subset Z_0U_i \]
and $Z_0 U_i$ is contained in either $Z_0X$ or $Z_0Y$. Hence, these elements have no unique representation in $TT$. If $U_i= Z_0$, the elements not identified within the table are $\{b^{-2^{k+1}},\;b^{2^{k+1}} \}$. Since we have
\[ b^{-2^{k+1}}=ab^{2^k+1}b^{2^k-1}a^{-1}\in y_{(0, 2^k+1)}X_{2^k-1} \] 
\[b^{2^{k+1}} = b^{2^k-1}a^{-1}ab^{2^k+1}\in x_{(2^k-1, 0)}Y_0.\] 
these elements also have no unique representation in $TT$.

We can extend this idea further to account for the remaining slices in $Z_0Y$ and $Z_0X$. 
As illustrated in figure \ref{fig:T1}, the slices we have yet to account for in the subproduct set $Z_0Y$
are  
\[z_{(0,-2^k)}Y_0 = \{ab^j \mid 2^k+1 \leq j \leq 2^{k+1}+1\}\subset Y_0Z_0\] 
and 
\[z_{(0,2^k)}Y_{2^k-1} = \{b^{2^k-1}ab^j \mid 1-2^k \leq j \leq 1\}\subset Y_{2^k-1}Z_0.\] 

Similarly, as illustrated in figure \ref{fig:T2}, the slices we have yet to account for in the subproduct set $Z_0X$ are subsets of the slices 
\[z_{(0,-2^k)}X_1= \{ba^{-1}b^j\mid 2^k \leq 2^{k+1}+1\} \subset X_1Z_0\] 
and 
\[z_{(0,2^k)}X_{2^k-1} = \{b^{2^k-1}a^{-1}b^j \mid -2^k \leq j\leq 1\} \subset X_{2^k-1}Z_0.\]
This accounts for all the subproduct sets of the form $Z_0U_i$ and $U_i Z_0$.

\subsection{Matching the Remaining Words. \\}
Now consider the circled slices illustrated in Figures \ref{fig:T1} and \ref{fig:T2}.
Using the defining relations and the progression length we will show that each word in the remaining slices $u_{(m,i)}W_n\subset U_m W_n$ is equivalent to a word in another product set. There are $5$ cases to consider for the equivalences. Let $\epsilon \in \{-1,\; 1\}$ and let $n\in \{0,\; 1,\; \dots 2^k+1 \}$ be considered only when a slice of that form $b^{n}a^{\epsilon_1} b^{m} a^{\epsilon_2} b^{i}$ exists.\\
{\bf Important Note:} In order to eliminate unnecessary redundancy, when we say that a range of words of a specified form occur in the product set $U_{m_1} W_{m_2}$, we will always mean that these words come from the products:
\[(b^{m_1} a^{\epsilon_1} b)(b^{m_2}a^{\epsilon_2}b^{i}) \text{ and } 
(b^{m_1} a^{\epsilon_1} b^{2^k+1})(b^{m_2}a^{\epsilon_2}b^{i}) \]
that occur in $U_{m_1} W_{m_2}$ by taking the exponent $i$ over all possible values in $W_{m_2}$.
The only notable exception is when $U_{m_1}=X_0$ and then only products of the first type are considered.\\ \\
{\bf Case 1:}
Consider words of the form $b^{n}a^{-\epsilon}b^{0} a^{\epsilon} b^{i}$, $b^{n}a^{-\epsilon}b^{2^k} a^{\epsilon} b^{i}$ or $b^{n}a^{-\epsilon}b^{2^{k+1}} a^{\epsilon} b^{i}$ that occur in the slices,
$x_{(n,0)}Y_0$, $x_{(0,1)}Y_{2^k-1}$, $y_{(n,2^k+1)}X_{2^k-1}$, and \\ $x_{(n,2^k+1)}Y_{2^k-1}$
(note the slice $x_{(n,2^k+1)}Y_{2^k-1}$ is only considered for $n\neq 0$).
The words can be written as:
\[b^{n}a^{-\epsilon}b^{0} a^{\epsilon} b^{i}= b^{i+n},\]
\[b^{n}a^{-\epsilon}b^{2^k} a^{\epsilon} b^{i}= b^{n}a^{-\epsilon} a^{\epsilon} b^{i-2^k}=b^{n+i-2^k},\]
\[b^{n}a^{-\epsilon}b^{2^{k+1}} a^{\epsilon} b^{i}= b^{n}a^{-\epsilon} a^{\epsilon} b^{i-2^{k+1}}=b^{n+i-2^{k+1}}.\]
Clearly these are contained in the product set $Z_0Z_0$ which contains every word $b^j$ where $\{-2^{k+1}\leq j \leq 2^{k+1}\}$. \\ \\
{\bf Case 2:} 
Consider words of the form $b^{n}a^{\epsilon}b a^{\epsilon} b^{i}$, that occur in the shortened slices $ x_{(n, 0)}X_1$ and the slices $y_{(n, 1)}Y_0$. These words can be written as:
\[b^{n}a^{\epsilon}b a^{\epsilon} b^{i} = b^{n}a^{\epsilon}a^{-2\epsilon}a^{2\epsilon}b a^{\epsilon} b^{i}
=b^{n}a^{\epsilon}a^{-2\epsilon}b a^{-2\epsilon}a^{\epsilon} b^{i} = b^{n}a^{-\epsilon}b a^{-\epsilon} b^{i}.\] 
In each slice $ x_{(n, 0)}X_1$, the exponent $i$ in the equivalent words has range \\
$\{2\leq j \leq 2^k+1\}$ after the shortening that is illustrated in Figure \ref{fig:T2}. Each product set $Y_nY_0$ contain words of the form $b^{n}a b a b^{j}$ where $\{1-2^k \leq j \leq 2^k+1\}$.
Whence $ x_{(n, 0)}X_1 \subset Y_nY_0$. 

Similarly, $y_{(n, 1)}Y_0 \subset X_nX_1$. Indeed the range of $j$-exponents for each equivalent word in $y_{(n, 1)}Y_0$ is $\{1\leq j \leq 2^k+1\}$ and the product sets $X_nX_1$ contain words $b^{n}a^{-1} b a^{-1} b^{j}$ where $\{0 \leq j \leq 2^k+1\}$ if $n=0$ and $\{-2^k \leq j \leq 2^k+1\}$ otherwise.\\ \\
{\bf Case 3:} 
Consider words of the form $b^{n}a b^{2} a^{-1} b^{i}$ that 
occur in the shortened slices $y_{(n,1)}X_1$. These words can be written as:
\[b^{n}a b^{2} a^{-1} b^{i}= b^{n}a a^{-2} a^{2}b^{2} a^{-1} b^{j}=b^{n}a a^{-2} b^{2} a^{2}a^{-1} b^{i}
=b^{n}a^{-1} b^{2} a b^{i}.\]
In each slice $y_{(n,1)}X_1$, the exponent $i$ in the equivalent words has range\\ $\{2 \leq i \leq 2^k-1\}$ after shortening. Each product set $X_nY_1$ contains words of the form $b^{n}a b^{2} a b^{j}$
where $\{1 \leq j \leq 2^k+1\}$ in the event $n=0$ and $\{1-2^k \leq j \leq 2^k+1\}$ otherwise. 
In all cases it follows then that $y_{(n,1)}X_1\in X_nY_1$. \\ \\
{\bf Case 4:}
Consider words of the form $b^{r}a^{\epsilon}b^{2^k} a^{\epsilon} b^{i}$ or $b^{r}a^{\epsilon}b^{2^{k+1}} a^{\epsilon} b^{i}$, where $r$ is even that occur in the slices $ x_{(0, 1)}X_{2^k-1}$, $ x_{(r, 2^k+1)}X_{2^k-1}$, and $y_{(r, 2^k+1)}Y_{2^k-1}$. These words can be written as: 
\[b^{r}a^{\epsilon}b^{2^k} a^{\epsilon} b^{i} =  b^{r}a^{2\epsilon} b^{i-2^k} = a^{2\epsilon} b^{r+i-2^k} \]
\[b^{r}a^{\epsilon}b^{2^{k+1}} a^{\epsilon} b^{i} =  b^{r}a^{2\epsilon} b^{j-2^{k+1}} = a^{2\epsilon} b^{r+i-2^{k+1}}. \] 
If $r\neq 0$, each slice $ x_{(r, 2^k+1)}X_{2^k-1}$ the exponent $j=r+i-2^{k+1}$ in the equivalent words has range 
$\{r-2^{k+1}\leq j \leq r+1-2^k\}$. Each product set $Y_{r-1}Y_{2^k-1}$ contains words of the form $a^{-2}b^j$ where $j$ has range $\{r-2^{k+1} \leq j \leq r \}$. 
Comparing lengths of the $j$ exponents, it follows that $ x_{(r, 2^k+1)}X_{2^k-1} \subset Y_{r-1}Y_{2^k-1}$.

In the case $r=0$, we only have the slice $ x_{(0, 1)}X_{2^k-1}$ where the exponent \\$j= i-2^k$ has range $\{-2^k \leq j \leq 1 \}$. The product set $Y_{2^k-1} Y_{2^k-1}$ contains products of the form $a^{-2}b^j$ where $j$ ranges between $\{-2^k \leq j \leq 2^k\}$. 
This shows that $ x_{(0, 1)}X_{2^k-1} \subset Y_{2^k-1} Y_{2^k-1}$.

In each slice $y_{(r, 2^k+1)}Y_{2^k-1}$, the exponent $j=r+i-2^{k+1}$ in the equivalent words has range $\{r+1-2^{k+1}\leq j \leq r+1-2^k\}$. Each product set $X_{r+1}X_{2^k-1}$
contains contains words of the form $a^2b^j$ where the exponent $j$
ranges between $\{ r+1-2^{k+1}\leq j \leq r+2\}$. 
This shows $y_{(r, 2^k+1)}Y_{2^k-1} \subset X_{r+1}X_{2^k-1}$.\\ \\
{\bf Case 5:}
Similarly, if $r$ is odd, say in the slices $x_{(r,2^k+1)}X_{2^k-1}$ and $y_{(r,2^k+1)}Y_{2^k-1}$ the following holds:
\[b^{r}a^{\epsilon}b^{2^{k+1}} a^{\epsilon} b^{i} =  b^{r}a^{2\epsilon} b^{i-2^{k+1}} = a^{-2\epsilon} b^{r+i-2^{k+1}}. \]
The arguments here are identical to the preceding case, the only case of interest is the slice $y_{(2^k-1, 2^k-1)}Y_{2^k-1}$. The range of exponents $j=i-1-2^k$ has range 
$\{-2^k\leq j \leq 0\}$. These words are contained in the product set $X_0 X_{2^k-1}$ which contains words of this form in the range $\{-2^k \leq j \leq 1\}$.

\begin{small}
\begin{table}[H]
\begin{center}
\begin{tabular}{|c|c|c|}
\hline
\multicolumn{3}{|c|}{Remaining Elements in $TT$} \\
\hline 
 Slice & Rewritten Elements & Remaining Values for $j$ \\ \hline
$x_{(0,0)}X_1$ &  {$abab^j \subset Y_0Y_0$} & {$2 \leq j \leq 2^k+1$} \\ \hline
$x_{(0,1)}X_{2^k-1}$ & {$a^{-2}b^j \subset Y_{2^k-1}Y_{2^k-1}$} & {$-2^k \leq j \leq 1$} \\ \hline
$x_{(l,0)}X_{1}$ & {$b^labab^j \subset Y_{l}Y_{0}$} &{$2 \leq j \leq 2^k+1$} \\ \hline
$x_{(l,2^k+1)}X_{2^k-1}$ &  {$a^2b^j \subset Y_{l-1}Y_{2^k-1}$} & {$l-2^{k+1} \leq j \leq l+1-2^k$} \\ \hline
$x_{(m,0)}X_{1}$ & {$b^mabab^j \subset Y_{m}Y_{0}$} &{$2 \leq j \leq 2^k+1$} \\ \hline
$x_{(m,2^k+1)}X_{2^k-1}$ &  { $a^{-2}b^j  \subset Y_{m-1}Y_{2^k-1}$} & {$m-2^{k+1} \leq j \leq m+1-2^k$} \\ \hline
$x_{(2^k-1,0)}X_{1}$ & {$b^{2^k-1 }abab^j \subset Y_{2^k-1}Y_0$} & {$2 \leq j \leq 2^k+1$} \\ \hline
$x_{(2^k-1,2^k+1)}X_{2^k-1}$ & { $a^{2}b^j \subset Y_{2^k-2}Y_{2^k-1}$} & {$-1-2^k \leq j \leq 0$} \\ \hline
$y_{(n,1)}X_{1}$ & {$b^na^{-1}b^{2}ab^j \subset X_{n}Y_{1}$} & {$2 \leq j \leq 2^k+1$} \\ \hline
$y_{(n,2^k+1)}X_{2^k-1}$ & { $b^j \subset Z_0Z_0$} & {$n-2^{k+1} \leq j \leq n+1-2^k$} \\ \hline
$y_{(0,1)}Y_{0}$ &  {$abab^j \subset X_{0}X_{1}$} & {$1 \leq j \leq 2^k+1$ } \\ \hline
$y_{(0, 2^k+1)}Y_{2^k-1}$ & {$a^2b^j \subset X_1X_{2^k-1}$} &
{$1-2^{k+1} \leq j \leq 1-2^k$} \\ \hline
$y_{(l,1)}Y_{0}$ & {$b^labab^j \subset  X_{l}X_{1}$ }& {$1 \leq j \leq 2^k+1$} \\ \hline
$y_{(l,2^k+1)}Y_{2^k-1}$ & { $a^{-2}b^j \subset X_{l+1}X_{2^k-1}$} & {$l+1-2^{k+1} \leq j \leq l+1-2^k$} \\ \hline
$y_{(m,1)}Y_{0}$ & {$b^mabab^j \subset X_{m}X_{1}$} & {$1 \leq j \leq 2^k+1$} \\ \hline
$y_{(m,2^k+1)}Y_{2^k-1}$ & {$a^{2}b^j \subset X_{m+1}X_{2^k-1}$} & {$m+1-2^{k+1} \leq j \leq m+1-2^k$} \\ \hline
$y_{(2^k-1,1)}Y_{0}$ & {$b^{2^k-1}abab^j \subset X_{2^k-1}X_1$} & {$1 \leq j \leq 2^k+1$} \\ \hline
$y_{(2^k-1,2^k+1)}Y_{2^k-1}$ & {$a^{-2}b^j \subset X_{0}X_{2^k-1}$} & {$-2^{k} \leq j \leq 0$} \\ \hline
$x_{(n,0)}Y_{0}$ & {$b^j \subset Z_0Z_0$} & {$n+1 \leq j \leq n+1+2^k$} \\ \hline
$x_{(0,1)}Y_{2^k-1}$ & {$b^j\subset Z_0Z_0$} & {$1-2^k \leq j \leq 1$} \\ \hline
$x_{(n,2^k+1)}Y_{2^k-1}$ & {$b^j \subset Z_0Z_0$ } & {$n+1-2^{k+1} \leq j \leq n+1-2^k$} \\ \hline
\end{tabular}
\end{center}
\label{}
\end{table}
\end{small}

The chart above summarizes the results from the preceding $5$ cases as a systematic listing of the circled slices illustrated in Figures \ref{fig:T1} and \ref{fig:T2}, where \\$l\in \{1,\; 3,\; \dots,\; 2^k-3\}$, $m\in \{2,\; 4,\; 6,\; \dots,\; 2^k-2\}$, and we set $n\in \{0,\; 1,\; \dots,\; 2^k-1\}$ to only be considered for all values where a slice of that form exists. This completes the proof of Theorem \ref{nup}.

\section{Cardinalities of Non-unique Product Sets}
 From the standpoint of Conjectures \ref{zdc} and \ref{nuc}, it seems natural to consider the cardinality of the possible non-unique product sets in $G$. Indeed if the cardinality of such sets were bounded, then one need only consider products in $k[G]$ of bounded support size. In this section, we will show that this is not possible in general, by showing that each $P_k$ contains arbitrarily large square non-unique product sets.  

The construction in the preceding section shows that $P_k$ contains a set $T$ with cardinality $2^{2k+1}+2^{k+2}+1$ having the property that $TT$ contains no uniquely represented elements. This section can be thought of as a corollary to the construction in Section $4$, in the sense that the features of the product sets we will construct in this section will mimic those in the preceding section. Indeed, the only meaningful distinction will be the lengths of the given progressions, and in the case where $p$ and $q$ are $1$, they will be the same. Given the similarities, we will omit an exhaustive analysis of the products. The analogous construction is done as follows 

Let $p$ be any fixed positive odd integer and choose an odd integer $q$ so that $q-1$ is a multiple of $2^k$. For these odd integers $p$ and $q$, consider the following $b$-progressions in $P_k$.
\[X_0(p, q) = \{a^{-p}b^{j}\;\mid\; -q+1 \leq j \leq (2^k+1)q-2^k \}, \]
\[X_i(p, q) = \{ b^{i} a^{-p}b^{j}\;\mid\; -q+1 \leq j \leq (2^k+1)q \},\]
\[Y_l(p, q) = \{ b^{l}a^{p}b^{j} \;\mid\; -q+2\leq j \leq (2^k+1)q\},\]
\[Z_0(p, q) = \{ b^j \;\mid\; -2^k(\frac{q+1}{2})-(q-1) \leq j \leq 2^k(\frac{q+1}{2})+(q-1) \} \]
where  $1\leq i \leq 2^k-1$ and $0\leq l \leq 2^k-1$.
We want to show that 
\[T(p, q) = \bigcup_{i=0}^{2^k-1} X_i(p, q) \cup \bigcup_{l=0}^{2^k-1} Y_l(p, q) \cup Z_0(p, q) \subset P_k\]
has the property that the product set $T(p, q) T(p, q)$ has no uniquely represented element. 

The reduction of each word to its normal form given by Proposition \ref{normal forms} shows that the words in $T(p, q)$ are distinct. 
Analogous to the construction in Section $4$, the majority of the words are matched using the structure of the progressions and the remaining cases are handled separately. By construction, the matching patterns illustrated in  \ref{fig:T1}, \ref{fig:T2}, and \ref{fig:T3} are identical. Moreover, adjusting the exponents appropriately shows that the extensions to the slices in the product sets $Z_0(p,q)X(p,q)$ and $Z_0(p,q)Y(p,q)$ also holds, i.e. both $-2^k(\frac{q+1}{2})-(q-1)$ and $2^k(\frac{q+1}{2})+(q-1)$ are multiples of $2^k$. Therefore, we need only consider the elements that are not matched via the progressions.

\begin{small}
\begin{center}
\begin{longtable}{|l|l|l|}

\hline
\multicolumn{3}{|c|}{Remaining Elements in $T(p, q) T(p, q)$} \\ \hline
 
\hline \multicolumn{1}{|c|}{ Slice} & \multicolumn{2}{c|}{Rewritten Elements } \\ \hline 

\hline \multicolumn{ 3}{|c|} {Remaining Values for $j$} \\  \hline
\endfirsthead
\hline
\multicolumn{3}{|c|}%
{{  -Continued from previous page}} \\
\hline \multicolumn{1}{|c|}{{Slice}} &
\multicolumn{2}{c|}{Remaining Elements}  \\ \hline 
\multicolumn{3}{|c|}{Remaining Values for $j$}  \\ \hline
\endhead
\hline \multicolumn{3}{|l|}{{Continued on next page-}} \\ \hline
\endfoot
\hline 
\endlastfoot
\hline $x_{(0,-q+1)}X_1$ & \multicolumn{ 2}{c|} {$a^{p}  ba^{p}b^j \subset Y_0Y_0$} \\ \hline
  \multicolumn{ 3}{|c|}{$ 2^kq+2q-2^k\leq j \leq 2^kq+2q-1$} \\ \hline
$x_{(0,(2^k+1)q-2^k)}X_{2^k-1}$ & \multicolumn{ 2}{c|} {$ a^{-2p}b^j \subset Y_{2^k-1}Y_{2^k-1)}$} \\ \hline
  \multicolumn{ 3}{|c|}{$ 2-2^kq-2q\leq j \leq 1 $} \\ \hline
$x_{(l,-q+1)}X_{1}$ & \multicolumn{ 2}{c|} {$b^l a^{p}ba^{p}b^j \subset Y_lY_0$} \\ \hline
  \multicolumn{ 3}{|c|}{$ 2^kq+2q-2^k\leq j \leq 2^kq+2q-1$} \\ \hline
$x_{(l,(2^k+1)q)}X_{2^k-1}$ & \multicolumn{ 2}{c|} {$a^{2p}b^j \subset Y_{l-1}Y_{2^k-1}$} \\ \hline
  \multicolumn{ 3}{|c|}{$ l+2-2^kq-2q-2^k\leq j \leq l+1-2^k $} \\ \hline
$x_{(m,-q+1)}X_{1}$ & \multicolumn{ 2}{c|} {$b^ma^{-p}   ba^{-p}b^j\subset Y_mY_0$} \\ \hline
  \multicolumn{ 3}{|c|}{$2^kq+2q-2^k \leq j \leq 2^k+2q-1$} \\ \hline
$x_{(m,(2^k+1)q)}X_{2^k-1}$ & \multicolumn{ 2}{c|} { $  a^{-2p}b^j \subset Y_{m-1}Y_{2^k-1}$} \\ \hline
  \multicolumn{ 3}{|c|}{$ m+2-2^kq-2q-2^k\leq j \leq m+1-2^k $} \\ \hline
$x_{(2^k-1,-q+1)}X_{1}$ &  \multicolumn{ 2}{c|} {$ b^{2^k-1}a^{p}ba^{p}b^j \subset Y_{2^k-1}Y_0$} \\ \hline
  \multicolumn{ 3}{|c|}{$ 2^kq+2q-2^k\leq j \leq 2^kq+2q-1 $} \\ \hline
$x_{(2^k-1,(2^k+1)q)}X_{2^k-1}$ &  \multicolumn{ 2}{c|} {$  a^{2p}b^j \subset Y_{2^k-2}Y_{2^k-1}$} \\ \hline
  \multicolumn{ 3}{|c|}{$ 1-2^kq-2q\leq j \leq 0 $} \\ \hline
$y_{(n,-q+2)}X_{1}$ & \multicolumn{ 2}{c|} {$b^na^{p}b^{2 }  a^{-p}b^j \subset X_nY_1$} \\ \hline
  \multicolumn{ 3}{|c|}{$2^kq+2q-2^k \leq j \leq 2^kq+2q-1$} \\ \hline
$y_{(n,(2^k+1)q)}X_{2^k-1}$ & \multicolumn{ 2}{c|} {$b^j \subset Z_0Z_0$} \\ \hline
  \multicolumn{ 3}{|c|}{$n+2-2^kq-2q-2^k \leq j \leq n+1 -2^k$} \\ \hline
$y_{(0,-q+2)}Y_{0}$ & \multicolumn{ 2}{c|} {$a^{p}b a^{p}b^j \subset X_0X_1$} \\ \hline
  \multicolumn{ 3}{|c|}{$ 1 \leq j \leq 2^kq+2q-1$} \\ \hline
$y_{(0,(2^k+1)q)}Y_{2^k-1}$ & \multicolumn{ 2}{c|}{$a^{2p}b^j \subset X_{1}X_{2^k-1}$} \\ \hline
  \multicolumn{ 3}{|c|}{$ 3-2^kq-2q-2^k \leq j \leq 1-2^k$} \\ \hline
$y_{(l,-q+2)}Y_{0}$ & \multicolumn{ 2}{c|} {$b^la^{p}b a^{p}b^j \subset X_lX_1$} \\ \hline
  \multicolumn{ 3}{|c|}{$ 1 \leq j \leq 2^kq+2q-1$} \\ \hline
$y_{(l,(2^k+1)q)}Y_{2^k-1}$ & \multicolumn{ 2}{c|} {$a^{-2p}b^j \subset X_{l+1}X_{2^k-1}$} \\ \hline
  \multicolumn{ 3}{|c|}{$l+3-2^kq-2q-2^k\leq j \leq l+1-2^k $} \\ \hline
$y_{(m,-q+2)}Y_{0}$ & \multicolumn{ 2}{c|} {$ b^ma^{p}ba^{p}b^j  \subset X_mX_1$} \\ \hline
  \multicolumn{ 3}{|c|}{$ 1\leq j \leq 2^kq+2q-1$} \\ \hline
$y_{(m,(2^k+1)q)}Y_{2^k-1}$ & \multicolumn{ 2}{c|} {$a^{2p}b^j \subset X_{m+1}X_{2^k-1}$} \\ \hline
  \multicolumn{ 3}{|c|}{$ m+3-2^kq-2q-2^k\leq j \leq m+1-2^k  $} \\ \hline
$y_{(2^k-1,-q+2)}Y_{0}$ & \multicolumn{ 2}{c|} {$b^{2^k-1}a^{p}ba^{p}b^j \subset X_{2^k-1}X_1$} \\ \hline
  \multicolumn{ 3}{|c|}{$1 \leq j \leq 2^q+2q-1$} \\ \hline
$y_{(2^k-1,(2^k+1)q)}Y_{2^k-1}$ & \multicolumn{ 2}{c|} {$a^{-2p}b^j \subset X_{0}X_{2^k-1}$} \\ \hline
  \multicolumn{ 3}{|c|}{$2-2^kq-2q \leq j \leq 0$} \\ \hline
$x_{(n,-q+1)} Y_{0}$ & \multicolumn{ 2}{c|} {$b^j \subset Z_0Z_0$} \\ \hline
  \multicolumn{ 3}{|c|}{$ n+1\leq j \leq n+2^kq+2q-1 $} \\ \hline
$x_{(0,(2^k+1)q-2^k)}Y_{2^k-1}$ & \multicolumn{ 2}{c|} {$b^j  \subset Z_0Z_0$} \\ \hline
  \multicolumn{ 3}{|c|}{$ 3-2^kq-2q\leq j \leq 1 $} \\ \hline
$x_{(n,(2^k+1)q)}Y_{2^k-1}$ & \multicolumn{ 2}{c|} {$b^j \subset Z_0Z_0$} \\ \hline
  \multicolumn{ 3}{|c|}{$n+3-2^kq-2q-2^k \leq j \leq n+1-2^k$} \\ 
\end{longtable}
\end{center}
\end{small}

The chart above contains a systematic listing of all the matchings for each of the remaining words, where $(p,q)$ are suppressed. Once again, we let \\ $l\in \{1,\; 3,\; 5,\; \dots,\; 2^k-3\}$, $m\in \{2,\; 4,\; 6,\; \dots,\; 2^k-2\}$, and we take \\ $n\in \{0,\; 1,\; \dots,\; 2^k-1\}$ to only be considered for all values where a slice of that form exists. One checks that the words in each slice are reduced in exactly the same way as they are in Section $4.2$. Therefore, we need only focus on the range of $j$ exponents of words contained the smaller product sets. 
Analogous to Section $4.2$, the range of words of a specified form that occur in the product sets $U_{m_1}W_{m_2}$ listed in the chart above come from the products 
\[(b^{m_1} a^{p\epsilon_1} b^{-q+2})(b^{m_2}a^{p\epsilon_2}b^{i}),
  \dots,
(b^{m_1} a^{p\epsilon_1} b^{(2^k+1)q-2^k})(b^{m_2}a^{p\epsilon_2}b^{i})\]
taking $i$ over all possible values in $W_{m_2}$. In the event that 
$U_m\neq X_0(p,q)$, we also include the products  $(b^{m_1} a^{p\epsilon_1} b^{(2^k+1)q})(b^{m_1}a^{p\epsilon_2}b^{i})$. 
The results follow by comparing the relative lengths listed for the slices in the chart above.

As a result of the containments, $T(p, q)$ is also a non-unique product set. Each set $T(p,q)\subset P_k$ has cardinality  $(2^{2k+1}+5 \times2^k +2)q-(2^k+1)$ which establishes Theorem \ref{unbound}. 
In our construction, we only needed that $p$ was an odd positive integer, if we consider 
\[ \{ T(2n-1,\; q) \mid\; n\geq 1 \text{ and } q-1 \text{ is a fixed multiple of }2^k \},\] 
this also shows there are infinitely many distinct square non-unique product sets for any fixed cardinality.\\ 

\emph{Acknowledgments:} I would like to thank Max Forester and Peter Linnell for their helpful comments and suggestions. I would like to also thank the anonymous referee for a careful reading.


\vspace{5mm}

Department of Mathematics, 
University of Oklahoma,
Norman, OK 73019, USA

\emph{Email:}\texttt{\lowercase{wcarter@math.ou.edu}}

\end{document}